\documentclass[11pt]{article}
\usepackage{amsfonts}
\usepackage{lipsum}
\usepackage{latexsym}
\usepackage{amsmath}
\usepackage{amssymb}
\usepackage{amsthm}
\usepackage{graphicx}
\usepackage{dsfont}
\usepackage{fullpage}
\usepackage{enumerate}
\usepackage{color}
\usepackage{bbm}
\newtheorem{theorem}{Theorem}[section]
\newtheorem{lemma}[theorem]{Lemma}

\newtheorem{proposition}[theorem]{Proposition}
\newtheorem{corollary}[theorem]{Corollary}

\newtheorem{claim}[theorem]{Claim}

\theoremstyle{definition}

\newtheorem{example}[theorem]{Example}

\newtheorem{thmy}{Theorem}
\newenvironment{oldtheorem}{\stepcounter{thm}\begin{thmy}}{\end{thmy}}
\newtheorem*{note*}{Note}

\newcommand{\R}{\mathbb R}

\makeatother

\newcommand\blfootnote[1]{%
  \begingroup
  \renewcommand\thefootnote{}\footnote{#1}%
  \addtocounter{footnote}{-1}%
  \endgroup
}

\begin{document}


\title{\bf A strengthening of the Blaschke-Santal\'o inequality for $o$-symmetric planar convex bodies}
\date{\today}
\medskip
\author{K\'aroly J. B\"or\"oczky, Endre Makai, Jr.}
\maketitle

\begin{center}
{\it Heartfully dedicated to K\'aroly Bezdek on the occasion of his 70th birthday}
\end{center}

\begin{abstract}
We verify the inequality
$$
\frac{|K|}{|E|}+\frac{|K^*|}{|E^*|}\leq 2
$$
 for any $o$-symmetric convex body $K\subset\R^2$ where $E$ is either the John ellipse of maximal area  contained in $K$ or the minimal area L\"owner ellipse containing $K$. The analogous estimate may not hold if $K$ is a planar but the assumption of $o$-symmetry is dropped, or if $K$ is $o$-symmetric convex body in $\R^n$ for $n\geq 3$. Our new inequality strengthens the Blaschke-Santal\'o inequality for  $o$-symmetric convex bodies $K\subset\R^2$ with an error term of optimal order.
\end{abstract}

\section{Introduction}

For background on the notions in convexity in the note, see Schneider \cite{Sch14}. We consider an inner product $\langle \cdot ,\cdot\rangle$ in $\R^n$, $n\geq 2$, and write $o$ to denote the origin and $\|x\|:=\sqrt{\langle x,x\rangle}$ to denote the Euclidean norm of an  $x\in \R^n$. In addition, let $B^n:=\{x\in \R^n:\|x\|\leq 1\}$ be the Euclidean unit ball   and  $S^{n-1}:=\{x\in\R^n:\|x\|=1\}$ be the unit sphere. The Lebesgue measure of a measurable $X\subset \R^2$ is denoted by $|X|$.  For an $o$-symmetric convex body $K\subset\R^n$, its polar is
$$
K^*=\{x\in\R^n:\langle x,y\rangle\leq 1\;\forall y\in K\},
$$
that satisfies $(B^n)^*=B^n$, $(K^*)^*=K$ and
\begin{equation}
\label{polar-linear-map}
(\Phi K)^*=\Phi^{-t}K^* 
\end{equation}
for any $\Phi\in {\rm GL}(n)$. A classical result in convex geometry is the
 Blaschke-Santal\'o inequality, which reads as follows in the case of an $o$-symmetric convex bodies.

\begin{oldtheorem}[Blaschke-Santal\'o inequality]
\label{thm-old-BS} 
If $K\subset\R^n$ is an $o$-symmetric convex body, then
\begin{equation}
\label{eq-san}
|K|\cdot |K^*|\leq |B^n|^2,
\end{equation}
with equality if and only if $K$ is an ellipsoid. 
\end{oldtheorem}

The $n=2,3$ cases of \eqref{eq-san} were due to Blaschke, and the inequality \eqref{eq-san} for all $n\geq 2$ is proved by Santal\'o \cite{San} in 1949. Actually, the equality case of \eqref{eq-san} was clarified by Petty \cite{Pet85} only in 1985.
Since then, various proofs of the Blaschke-Santal\'o inequality have been  presented, for example, by \cite[Theorem 7.3]{And06},  \cite{Bal86}, \cite{Leh09b}, \cite{Leh09c}, \cite{Lut91}, \cite{LuZ97}, \cite{MeP90}.

It follows from \eqref{polar-linear-map} that the left hand side of \eqref{eq-san} does not change if $K$ is replaced by $\Phi K$ for a linear transform $\Phi\in{\rm GL}(n)$. In particular, one may assume that the unit ball $B^n$ is either the unique so-called John ellipsoid of largest volume contained in $K$, or the unique so-called L\"owner ellipsoid of smallest volume containing $K$ (see John \cite{Joh48}, Ball \cite{Bal92} or Gruber, Schuster \cite{GrS05} about the existence and uniqueness of these ellipsoids). According to the characterization of the contact points due to John \cite{Joh48} (see also Gruber, Schuster \cite{GrS05}), assuming that $B^n\subset K$ for an $o$-symmetric convex body $K\subset\R^n$, $B^n$ is the John ellipsoid of $K$ if and only if there exist $c_1,\ldots,c_k>0$ and an $o$-symmetric set $\{u_1,\ldots,u_k\}\subset S^{n-1}\cap\partial K$ such that
\begin{align}
\label{John-cond}
\sum_{i=1}^kc_i u_i\otimes u_i=&I_n\\
\label{John-card}
2n\leq k\leq & n(n+1) 
\end{align}
where $u_i\otimes u_i$ is the rank one $n\times n$ matrix $u_iu_i^t$. Similarly,  assuming that assuming that $C\subset B^n$ for an $o$-symmetric convex body $C\subset\R^n$, $B^n$ is the L\"owner ellipsoid of $C$ if and only if there exist $c_1,\ldots,c_k>0$ and an $o$-symmetric set $u_1,\ldots,u_k\in S^{n-1}\cap\partial K$ such that \eqref{John-cond} and \eqref{John-card} hold. It follows via \eqref{polar-linear-map} that for any $o$-symmetric convex body $K\subset\R^n$ and $o$-symmetric ellipsoid $E\subset\R^n$
\begin{equation}
\label{John-Lowner-polarity}
\mbox{$E$ is the John ellipsoid of $K$ if and only if $E^*$ is the L\"owner ellipsoid of $K^*$.}
\end{equation}
These results about the John and L\"owner ellipsoids  were implicit in Behrend \cite{Beh37} from 1937 in the planar case $n=2$.

Our main goal is to prove the following strengthening of the Blaschke-Santal\'o inequality in the planar case. 

\begin{theorem}
\label{thm-volume-sum}
If $K\subset\R^2$ is an $o$-symmetric convex body, such that either  its John ellipse  or its  L\"owner ellipse is the Euclidean unit ball $B^2$, then
\begin{equation}
\label{eq-volume-sum}
    |K|+ |K^*|\leq 2\pi,
\end{equation}
and equality holds if $K=B^2$.
\end{theorem}

We note that Theorem~\ref{thm-volume-sum} does not hold if $K$ is the regular triangle incribed into $B^2$. In this case, $B^2$ is the  minimal volume L\"owner ellipse by the uniquenes and the linear invariance of the L\"owner ellipse (see  John \cite{Joh48} or Gruber, Schuster \cite{GrS05}), and also $B^2$ is the John ellipse of the circumscibed regular triangle $K^*$. It follows that
$|K|+ |K^*|=6.49\ldots> 2\pi$.
On the other hand, no analogue of Theorem~\ref{thm-volume-sum} holds in higher dimensions, as if $n\geq 3$ and $K=[-1,1]^n$, then $|K|+ |K^*|>2|B^n|$. For $n=3$, this can be seen by direct calculations, and if $n\geq 4$, then already $|K|\geq 16>2|B^n|$.

Combining Theorem~\ref{thm-volume-sum}, \eqref{polar-linear-map} and \eqref{John-Lowner-polarity}, we deduce the following result that holds for any $o$-symmetric convex domain.

\begin{corollary}
\label{thm-volume-sum-gen}
If $K\subset\R^2$ is an $o$-symmetric convex body, and $E$ is either  its John ellipse  or its  L\"owner ellipse, then
\begin{equation}
\label{eq-volume-sum-gen}
    \frac{|K|}{2|E|}+ \frac{|K^*|}{2|E^*|}\leq 1.
\end{equation}
\end{corollary}

As a reverse inequality to Theorem~\ref{thm-volume-sum}, if a convex set $K\subset \R^2$  is either contains the unit disk $B^2$, or is contained in $B^2$ and contains the origin in their interior, then Florian \cite{Flo95,Flo96}  showed that 
\begin{equation}
\label{volume-sum-lower}
    |K|+ |K^*|\geq 6,
\end{equation}
with equality if and only if $K$ is a square circumscribed around or inscribed into $B^2$.

As $|E|\cdot|E^*|=\pi^2$ by \eqref{polar-linear-map}, Corollary~\ref{thm-volume-sum-gen} yields the planar Blaschke-Santal\'o inequality \eqref{eq-san} by the AM-GM inequality between the arithmetic and geometric mean. In addition, we deduce the  following stability version of the planar Blaschke-Santal\'o inequality. We note that according to Behrend \cite{Beh37}, p. 726, if  $E_J\subset K$ is the John ellipse  and $E_L\supset K$ is the L\"owner ellipse of the $o$-symmetric convex body $K\subset\R^2$, then  
\begin{equation}
\label{Behrend}
 1\geq \frac{|E_J|}{|K|}\geq \frac {\pi}4 \mbox{ \ and \ } 1\leq \frac{|E_L|}{|K|}\leq \frac{\pi}2
\end{equation}
where equality holds in the inequalities on the left for ellipses, and on the right for parallelograms (see Ball \cite{Bal89} for an extension of \eqref{Behrend} to any dimension). Stability versions of the Blaschke Santal\'o inequality in terms of the Banach-Mazur distance (that is weaker than the symmetric difference metric) have been obtained by Ball, B\"or\"oczky \cite{BaB11}, B\"or\"oczky \cite{Bor10} and Ivaki \cite{Iva15}.

\begin{theorem}
\label{BS-stab}
Let $K\subset\R^2$ be an $o$-symmetric convex body with John ellipse $E_J\subset K$ and  L\"owner ellipse $E_L\supset K$.
If 
$$
|K|\cdot  |K^*|\geq (1-\varepsilon)\pi^2
$$
for $\varepsilon\in[0,\frac12)$, then
\begin{equation}
\label{eq-BS-stab}
    |K\backslash E_J|\leq 4 |K|\sqrt{\varepsilon}\mbox{ and }|E_L\backslash K|\leq 5|K|\sqrt{\varepsilon}.
\end{equation}
\end{theorem}

\begin{example}[The order $\sqrt{\varepsilon}$ of the error term in \eqref{eq-BS-stab} is optimal]\mbox{ }\\
\label{BS-stab-opt}
For $p>1$, we consider the $L_p$ ball
$$
B^2_p=\{(x,y)\in\R^2:|x|^p+|y|^p\leq 1\},
$$
and hence $B^2_2=B^2$. We recall that according to Wang \cite{Wan05}, we have
\begin{equation}
\label{b2p-vol}
|B^2_p|=\frac{4\cdot \Gamma(1+\frac1p)^2}{\Gamma(1+\frac{2}p)}
\end{equation}
where $\Gamma(\cdot)$ is Euler's Gamma function.
If $|t|$ is small, then let $K_t=B^2_p$ for  $p=2+t$, and hence \eqref{John-cond} yields that the John ellipsoid of $K_t$ is $B^2$ if $t\geq 0$, and the L\"owner ellipsoid of $K_t$ is $B^2$ if $t\leq 0$. In addition, $K^*_t=B^2_q$ holds for $q=(2+t)/(1+t)$.

We observe that
\begin{align}
\label{Kt-voldiff-tpos}
|K_t\backslash B^2|\geq &c_1 t&&\mbox{if }t\geq 0,\\
\label{Kt-voldiff-tneg}
|B^2\backslash K_t|\geq &c_1 t&&\mbox{if }t\leq 0
\end{align}
for an absolute constant $c_1>0$. 
For the $C^2$ function $f(t)=|K_t|\cdot |K_t^*|$, combining $K_0=B^2$ and the Blaschke-Santal\'o inequality \eqref{eq-san} implies that $f$ attains its maximum at $t=0$, and hence $f'(0)=0$, and, in turn, the Taylor formula yields that
\begin{equation}
\label{f-near-zero}
f(t)\geq f(0)-c_2t^2
\end{equation}
for an absolute constant $c_2>0$. We conclude  from \eqref{Kt-voldiff-tpos} and \eqref{f-near-zero} that if $\varepsilon>0$ is small, and we choose $t=\pm\sqrt{\varepsilon/c_2}$, then 
$$
|K_t|\cdot |K_t^*|\geq (1-\varepsilon)\pi^2,
$$
while $|K_t\backslash B^2|\geq \frac{c_1}{\sqrt{c_2}}\cdot\sqrt{\varepsilon}$ if $t>0$, and $|B^2\backslash K_t|\geq \frac{c_1}{\sqrt{c_2}}\cdot\sqrt{\varepsilon}$ if $t<0$. This completes the proof of Example~\ref{BS-stab-opt}.
\end{example}

\section{Area sum within a convex cone}
\label{secSector}

For $X_1,\ldots,X_m\subset\R^2$, we write $[X_1,\ldots,X_m]$ to denote the convex hull of $X_1\cup\ldots\cup X_m$; namely, the smallest convex set containing $X_1,\ldots,X_m$. Here $[X_1,\ldots,X_m]$ is compact if $X_1,\ldots,X_m$ are compact. For $x\neq y\in\R^2$, we write ${\rm aff}\{x,y\}$ to denote the line passing through $x,y$, and write
$$
(x,y)=[x,y]\backslash\{x,y\}
$$
to denote the open segment. If $K,M\subset \R^2$ are convex bodies, then their Hausdorff distance is
$$
\delta_H(K,M)=\min_{\varrho\geq 0}\{K\subset M+\varrho B^2\mbox{ and }M\subset K+\varrho B^2\}.
$$
It follows that if $M_n\subset \R^2$ is a sequence of convex bodies, then
\begin{equation}
\label{limit}
\lim_{n\to\infty}\delta_H(M_n,M)=0\mbox{ \ if and only if \ }\lim_{n\to\infty}|M_n\Delta M|=0.
\end{equation}
In this case, we say that $M_n$ tends to $M$ with respect to the Hausdorff metric. According to the Blaschke Selection theorem, if 
for fixed $R>r>0$, $M_n\subset RB^2$ is a sequence of convex bodies  such that  $x_n+rB^n\subset M_n$ for some $x_n\in M_n$, then there exists a subsequence $\{M_{n'}\}$ tending to some convex body with respect to the Hausdorff metric.
One of our tools in this section is Steiner symmetrization.
For a convex body $M\subset \R^2$, and
a line $\ell$ through $o$, translate every secant of $M$
orthogonal to $l$ in its affine hull in a way
such that the midpoint of the translated image lies on $\ell$.
The closure of the union of these translates
is the Steiner symmetral $M'$ with respect to $\ell$. 
Readily $M'$ is a convex body with $|M'|=|M|$.
According to Keith Ball's PhD thesis \cite{Bal86} (see also Meyer, Pajor \cite{MeP90}), if $K\subset \R^2$ is an $o$-symmetric planar convex body and $K'$ is the Steiner symmetral of $K$ with respect to $\ell$, then
\begin{equation}
\label{Steiner-symm}
|K'^*|\geq |K^*|.
\end{equation}

Polarity is frequently considered for a $z\in\R^2\backslash\{o\}$ or a line $\ell\subset \R^2\backslash\{o\}$, and is defined by
\begin{align*}
z^*=&\{x\in \R^2:\langle x,z\rangle=1\}\\
\ell^*=&\{x\in \R^2:\langle x,y\rangle=1\mbox{ for any }y\in\ell\}.
\end{align*}
It follows that $z^*$ is a line in $\R^2\backslash\{o\}$ such that $z$ is a normal vector and  $(z^*)^*=z$, and $\ell^*$ is a normal vector to $\ell$ satisfying $(\ell^*)^*=\ell$. In addition, if $K\subset\R^2$ is an $o$-symmetric convex body and $\Phi\in{\rm GL}(2)$, then
\begin{equation}
\label{point-line-polarity}
\begin{array}{rl}
z\in\ell&\mbox{ if and only if \ }\ell^*\in z^*,\\
(\Phi z)^*=\Phi^{-t}z^*&\mbox{ and \ }(\Phi \ell)^*=\Phi^{-t}\ell^*,\\
z\in\partial K&\mbox{ if and only if \ }z^*\mbox{ is a supporting line to }K^*.
\end{array}
\end{equation}

Next, we set up the notation used in Proposition~\ref{BS-inC}. For fixed $p,q\in S^1$ with $p\neq \pm q$, let $r$ be the intersection of the tangent lines of $B^2$ at $p$ and $q$. We consider the deltoid $C_{p,q}=[o,p,r,q]$ and the convex cone $\sigma=\{tp+sq:s,t\geq 0\}$, and the family
$$
\mathcal{K}_{p,q}=\{M\subset C_{p,q}: M\mbox{ is a convex body and }o,p,q\in M\}.
$$
To any $M\in\mathcal{K}_{p,q}$, we associate the $o$-symmetric convex body $K_M=[M,-M]$, and the "relative polar"
$$
M^\circ=\{x\in \sigma:\langle x,y\rangle\leq 1\mbox{ for any }y\in M\}\subset C_{p,q},
$$
and hence \eqref{point-line-polarity} yields that
\begin{equation}
\label{relative-polar}
M^\circ=K_M^*\cap \sigma\in \mathcal{K}_{p,q} \mbox{ \ and \ } (M^\circ)^\circ=M.
\end{equation}
In addition, if $M_n\in\mathcal{K}_{p,q}$ is a sequence of convex bodies tending to a convex body $M\subset \R^n$ with respect to the Hausdorff metric as $n$ tends to infinity, then $M\in\mathcal{K}_{p,q}$ and
\begin{equation}
\label{relative-polar-limit}
\lim_{n\to\infty}M_n^\circ=M^\circ\mbox{ \ and \ }\lim_{n\to\infty}|M_n^\circ|=|M^\circ|.
\end{equation}

\begin{proposition} 
\label{BS-inC}
For fixed $p,q\in S^1$ with $0<\angle(p,o,q)\leq \frac{\pi}2$, if $M\in\mathcal{K}_{p,q}$, then
$$
|M|+|M^\circ|\leq \angle(p,o,q).
$$
\end{proposition}
  
We prove some statements to prepare the proof Proposition~\ref{BS-inC}. For $p,q\in S^1$ with $0<\angle(p,o,q)\leq \frac{\pi}2$, let $\ell_{p,q}$ be the line passing through $o$ and $\frac{p+q}2$, and hence $C_{p,q}$ is symmetric through $\ell_{p,q}$. We observe that $C_{p,q}^\circ=[0,p,q]$.
We classify an element  $M\in \mathcal{K}_{p,q}$  symmetric through $\ell_{p,q}$ into two types; namely, $M$ is
\begin{description}
\item[type A,]  if $[p,r]\cap M$ (and hence $[q,r]\cap M$, as well) is a segment;
\item[type B,]  if  $[p,r]\cap M=\{p\}$ and $[q,r]\cap M=\{q\}$.
\end{description}
We observe that
\begin{equation}
\label{typea-polar-typeB}
\mbox{if $M\in \mathcal{K}_{p,q}$ is symmetric through $\ell_{p,q}$ and is of type A, then $M^\circ$ is of type B.}
\end{equation}

\begin{claim}
\label{Steiner-symm-inC}
If $M\in \mathcal{K}_{p,q}$, and $M'$ is the Steiner symmetral of $M$ with respect to $\ell_{p,q}$, then $M'\in \mathcal{K}_{p,q}$ and
\begin{equation}
\label{Steiner-symm-inC-eq}
|M'^\circ|\geq |M^\circ|.
\end{equation}
\end{claim}
\begin{proof} Since $C_{p,q}$ is symmetric through $\ell_{p,q}$, we have $M'\in \mathcal{K}_{p,q}$. 

We consider the $o$-symmetric convex body $K=[M,-M]$, and  let $f$ be the polar of ${\rm aff}\,\{p,-q\}$ and let $\Sigma$ be the closed strip bounded by the parallel lines ${\rm aff}\,\{p,q\}$ and ${\rm aff}\,\{-p,-q\}$. Since \eqref{point-line-polarity} yields that
$K^*=[M^\circ,-M^\circ,f,-f]$, $K'^*=[M'^\circ,-M'^\circ,f,-f]$ and $K'^*\cap \Sigma=[\pm p,\pm, q,\pm f]=K^*\cap \Sigma$, we deduce from \eqref{Steiner-symm} that
$$
2|M'^\circ|-2\big|[0,p,q]\big|=\left|K'^*\backslash \Sigma\right|\geq \left|K^*\backslash \Sigma\right|=2|M^\circ|-2\big|[0,p,q]\big|,
$$
proving \eqref{Steiner-symm-inC-eq}.
\end{proof}

The next claim states that $C_{p,q}$ or $C_{p,q}^\circ=[o,p,q]$ can't be optimal.

\begin{claim}
\label{C-not-optimal}
For $p,q\in S^1$ with $0<\angle(p,o,q)\leq \frac{\pi}2$,
there exists a type A pentagon $P\in \mathcal{K}_{p,q}$ symmetric through $\ell_{p,q}$ such that
$$
|C_{p,q}|+|C_{p,q}^\circ|<|P|+|P^\circ|.
$$
\end{claim}
\begin{proof}
For small $t>0$, we consider the line $\ell_t=\{x\in\R^2:\langle x,r\rangle= \langle r,r\rangle-t\}$ and the type A pentagon
$P_t=\{x\in C_{p,q}:\langle x,r\rangle\leq \langle r,r\rangle-t\}\in \mathcal{K}_{p,q}$ symmetric through $\ell_{p,q}$, and hence $P_t^\circ=[o,p,q,\ell_t^*]$, and
there exist constants $\gamma,\delta>0$ such that $|P_t|=|C_{p,q}|-\gamma t^2$ and $|P_t^*|=|C_{p,q}^*|+\delta t$. In turn, we conclude Claim~\ref{C-not-optimal}.
\end{proof}

The following simple well-known observation is the basis of our argument for Lemma~\ref{areasum-polygon-variation}. 

\begin{claim}
\label{minimal-area-cut}
Let $\tau=[h_1,h_2]$ be a convex cone bounded by the two non-collinear half-lines $h_1$ and $h_2$ meeting at the apex $a$ of $\tau$, and let $m\in{\rm int}\,\tau$. We write $\ell_0$ to denote the unique line passing through $m$ such that $m$ is the midpoint of the segment $\ell_0\cap\tau$; namely, $b_i=\ell_0\cap h_i$ is the intersection of $h_i$ and the reflected image of $h_{3-i}$ through $m$ for $i=1,2$. In addition, let $\ell_i$ be the line through $m$ parallel to $h_i$ and intersecting $h_{3-i}$ in $c_{3-i}$ for $i=1,2$.

If the line $\ell$ through $m$ is rotated from the position $\ell_2$ parallel to $h_2$ towards $\ell_0$ in a way such that $\ell$ intersects $h_{1}$ in a point $d_{1}$ moving from $c_{1}$ towards $b_{1}$, then the area of the triangle cut off from $\tau$ by $\ell$ is strictly decreasing. In particular, among lines passing through $m$, $\ell_0$ cuts off the smallest area triangle from $\tau$.
\end{claim}
\begin{proof}
Let $d_1, d'_1\in (c_1,b_1)$ such that $d'_1\in(d_1,b_1)$, let  $\ell={\rm aff}\{m,d_1\}$ and $\ell'={\rm aff}\{m,d'_1\}$. If $d_2=\ell\cap h_2$ and $d'_2=\ell'\cap h_2$, then Claim~\ref{minimal-area-cut} is equivalent to proving that $\big|[a,d_1,d_2]\big|> \big|[a,d'_1,d'_2]\big|$, which is, in turn, is equivalent to the statement
\begin{equation}
\label{ell-rotated-to-ellprime}
\big|[m,d_1,d'_1]\big|< \big|[m,d_2,d'_2]\big|.
\end{equation}
However, $\|d_1-m\|<\|d_2-m\|$ as $d_1$ is closer to $\ell_2$ than $b_1$, and, in turn, closer to $\ell_2$ than $d_2$. As similarly $\|d'_1-m\|<\|d'_2-m\|$ and the angles of the triangles $[m,d_1,d'_1]$ and $[m,d_2,d'_2]$ at $m$ coincide, we conclude \eqref{ell-rotated-to-ellprime}.
\end{proof}

Next, we consider the variation of the "area sum" when a polygon $P\in\mathcal{K}_{p,q}$  is suitably deformed. 

\begin{lemma}
\label{areasum-polygon-variation}
For $p,q\in S^1$ with $0<\angle(p,o,q)\leq \frac{\pi}2$,
let $u,v,w$ be consecutive vertices of a polygon $P\in\mathcal{K}_{p,q}$ of at least $5$ vertices such that $[v,w]$ intersects ${\rm int}\,C_{p,q}$, and let $\tilde{u}\in [u,v]\backslash\{v\}$ be such that the intersection $s$ of $[o,v]$ and  $[\tilde{u},w]$ lies in $(\tilde{u},\mu)$ for $\mu=(\tilde{u}+w)/2$. In addition, let $P_0\in\mathcal{K}$  be the closure of $P\backslash [\tilde{u},v,w]$. 

If $\tilde{v}-v$ is a non-zero vector parallel to $\ell_w={\rm aff}\{\tilde{u},w\}$ for $\tilde{v}\in{\rm int}\,C_{p,q}$ such that the intersection $\tilde{s}$ of $[o,\tilde{v}]$ and  $\ell_w$ lies in $(s,\mu)$ (and hence "$v$ is moved parallel to $\ell_w$ towards $w$" into the position $\tilde{v}$), and the polygon $\widetilde{P}=[\tilde{v},P_0]\in\mathcal{K}$ still has $w$ as a vertex, then
\begin{equation}
\label{areasum-polygon-variation-eq}
|\widetilde{P}|+|\widetilde{P}^\circ|>|P|+|P^\circ|.
\end{equation}
\end{lemma} 
\begin{proof}
For the line $\ell_v={\rm aff}\{v,\tilde{v}\}$ parallel to $\ell_w$, we observe that $\ell_v^*\in(0,\ell_w^*)\subset {\rm int}\,P_0^\circ$, and $P^\circ$ and $\widetilde{P}^\circ$ are obtained from $P_0^\circ$ by $v^*$ and $\tilde{v}^*$, respectively, cut off the vertex $\ell_w^*$ of $P_0^\circ$ where both $v^*$ and $\tilde{v}^*$ intersect both sides of $P_0^\circ$ containing $\ell_w^*$; namely, the ones contained in $\tilde{u}^*$ and $w^*$. We note that the point $\ell_v^*\in {\rm int}\,P_0^\circ$ is contained in the lines $v^*$ and $\tilde{v}^*$, and  writing $\tilde{\mu}=\ell_v\cap {\rm aff}\,\{0,\mu\}$, the point $\ell_v^*$ is the midpoint of the intersection of the line $\tilde{\mu}^*$ with $P_0^\circ$ where again  $\tilde{\mu}^*$ intersects both sides of $P_0^\circ$ containing $\ell_w^*$. Since $\tilde{v}\in (v,\tilde{\mu})$, we deduce from Claim~\ref{minimal-area-cut} that $|\widetilde{P}^\circ|>|P^\circ|$. As readily $|\widetilde{P}|=|P|$, we conclude \eqref{areasum-polygon-variation-eq}.
\end{proof}

Observe that in Lemma~\ref{areasum-polygon-variation}, we have estimated $|P|+|P^{\circ }|$ from above, by
fixing $|P|$, and increasing $|P^{\circ }|$. In the proof of Proposition~\ref{BS-inC}, we will often apply this method. However, we also will often
change the roles of $P$ and $P^{\circ }$. This amounts to estimate the
sum of areas from above, by fixing the area of the polar set, and
increasing the area of the original set.

The final auxiliary statement yields that $p$ and $q$ can be assumed orthogonal.

\begin{claim}
\label{pq-orthogonal}
For $p,q\in S^1$ with $0<\alpha=\angle(p,o,q)< \frac{\pi}2$,
 let $q'\in S^1$ be orthogonal to $p$ with $\langle q,q'\rangle>0$. For $N\in\mathcal{K}_{p,q'}$, we write
$$
N^{\circ'}=\{x\in C_{p,q'}:\langle x,y\rangle\leq 1\mbox{ for }y\in N\}.
$$
If $M\in\mathcal{K}_{p,q}$ and $\aleph=\{sq+tq'\in B^2:s,t\geq 0\}$, then $M\cup\aleph\in\mathcal{K}_{p,q'}$, and
\begin{equation}
\label{pq-orthogonal-eq}
|M\cup\aleph|+\big|(M\cup\aleph)^{\circ'}\big|=|M|+|M^\circ|+\frac{\pi}2-\alpha.
\end{equation}
\end{claim}
\begin{proof}
Since the tangent line to $B^2$ at $q$ is a supporting line to both $M$ and $\aleph$, we deduce that $M\cup\aleph$ is convex, and hence $M\cup\aleph\in\mathcal{K}_{p,q'}$. Moreover, \eqref{point-line-polarity} yields that $(M\cup\aleph)^{\circ'}=M^\circ\cup\aleph$, proving \eqref{pq-orthogonal-eq}.
\end{proof}

\begin{proof}[Proof of Proposition~\ref{BS-inC}]
According to Claim~\ref{pq-orthogonal}, we may assume that $p,q\in S^1$ are orthogonal in Proposition~\ref{BS-inC}.

During the proof of Proposition~\ref{BS-inC}, we approximate any element $M\in\mathcal{K}_{p,q}$ symmetric through $\ell_{p,q}$ by polygons symmetric through $\ell_{p,q}$. \\

\noindent{\bf Step 1.} {\it Properties (i), (ii) and (iii) of some extremal polygons in $\mathcal{K}_{p,q}$.}

For $m\geq 0$, we write $\mathcal{P}^{(m)}$ to denote the family of polygons $P\in \mathcal{K}_{p,q}$  that are symmetric through $\ell_{p,q}$ and have exactly $m$ sides that intersect ${\rm int}\,C_{p,q}$. In particular,  the only element of $\mathcal{P}^{(0)}$ is $C_{p,q}$. For $m\geq 1$, we call any endpoint of a side of a $P\in \mathcal{P}^{(m)}$, $P\neq C_{p,q}$, intersecting ${\rm int}\,C_{p,q}$ a proper vertex of $P$. 
For $m\geq 1$, let $\mathcal{P}^{(m)}_A$ and $\mathcal{P}^{(m)}_B$ be the family of elements in $\mathcal{P}^{(m)}$  of type A and type B, respectively, and hence we deduce from \eqref{point-line-polarity} that
\begin{equation}
\label{type-AB-polar}
P\in \mathcal{P}^{(m)}_A \mbox{ \ if and only if \ } P^\circ\in \mathcal{P}^{(m+1)}_B.
\end{equation}
For $m\geq 1$, we write $\widetilde{\mathcal{P}}^{(m)}=\bigcup_{i=0}^m\mathcal{P}^{(i)}$, $\widetilde{\mathcal{P}}^{(m)}_A=\bigcup_{i=0}^m\mathcal{P}^{(i)}_A$ and $\widetilde{\mathcal{P}}^{(m)}_B=\bigcup_{i=0}^m\mathcal{P}^{(i)}_B$, and hence if $P_n\in \widetilde{\mathcal{P}}^{(m)}_A$ is a sequence of polygons tending to a convex body $Q$, then 
\begin{equation}
\label{PmA-limit}
Q\in \widetilde{\mathcal{P}}^{(m)},
\end{equation}
and if $P_n\in \widetilde{\mathcal{P}}^{(m+1)}_B$ is a sequence of polygons tending to a convex body $Q$, then 
\begin{equation}
\label{Pm+1B-limit}
\mbox{either }Q\in \widetilde{\mathcal{P}}^{(m+1)}_B \mbox{ \ or \ }Q\in \widetilde{\mathcal{P}}^{(m-1)}_A.
\end{equation}
For $m\geq 6$, \eqref{type-AB-polar}, \eqref{PmA-limit} and \eqref{Pm+1B-limit} yield that there exists a polygon $Q_{(m)}\in \mathcal{K}_{p,q}$ such that either $Q_{(m)}\in \widetilde{\mathcal{P}}^{(m)}_A$ or $Q_{(m)}\in \widetilde{\mathcal{P}}^{(m+1)}_B$, and
$$
\left|Q_{(m)}\right|+\left|Q_{(m)}^\circ\right|=\sup\big\{|P|+|P^\circ|:\mbox{ either }P\in \mathcal{P}^{(m)}_A \mbox{  or } P\in \mathcal{P}^{(m+1)}_B\big\}.
$$
It follows from Claim~\ref{C-not-optimal} that $Q_{(m)}\neq C_{p,q}$ and $Q_{(m)}\neq [o,p,q]$. We may assume that the vertices $o,p,r,q$ of $C_{p,q}$ are in clockwise order in this order, and let $x_0,\ldots,x_k$ be the 
proper vertices of $Q_{(m)}$ in clockwise order along $\partial Q_{(m)}$ where $1\leq k\leq m$ and $x_0\in[p,r]\backslash\{r\}$. We claim that
\begin{description}
\item{(i)} $k\geq m-1$; 
\item{(ii)} for any $i=1,\ldots,k-1$, 
the points $o$, $x_i$ and $(x_{i-1}+x_{i+1})/2$ are collinear;

\item{(iii)} for any $i=2,\ldots,k-1$, if $\eta_i$ is the intersection point of ${\rm aff}\{x_{i-1},x_{i-2}\}$
and ${\rm aff}\{x_{i},x_{i+1}\}$,
then $o$, $\eta_i$ and $(x_{i-1}+x_i)/2$ are collinear. 

\end{description}
By the symmetry of the roles of $Q_{(m)}$ and $Q_{(m)}^\circ$, we also claim that the analogues of (i), (ii) and (iii) hold for the proper vertices of $Q_{(m)}^\circ$.

The main tool to verify (i) and (ii) is Lemma~\ref{areasum-polygon-variation}.
For (i), we suppose that $k\leq m-2$, and seek a contradiction. Possibly replacing $Q_{(m)}$ by $Q_{(m)}^\circ$ (cf. \eqref{type-AB-polar}), we may asssume that $Q_{(m)}$ is of type $A$; namely, 
$$
Q_{(m)}\in \widetilde{\mathcal{P}}^{(m-2)}_A.
$$ 
We choose an $x_{-1}\in(p,x_0)$ very close to $x_0$ in a way such that  such that the intersection $s_0$ of $[o,x_0]$ and  $[x_{-1},x_1]$ lies in $(x_{-1},\mu_0)$ for $\mu_0=(x_{-1}+x_1)/2$. We slightly move $x_0$ parallel to ${\rm aff}\{x_{-1},x_1\}$ into a position $\tilde{x}_0\in{\rm int}\,C_{p,q}$ in a way such that  the intersection $\tilde{s}_0$ of $[o,\tilde{x}_0]$ and  $[x_{-1},x_1]$ lies in $(s,\mu_0)$. We deduce from Lemma~\ref{areasum-polygon-variation} that 
$$
\left|R\right|+\left|R^\circ\right|>\left|Q_{(m)}\right|+\left|Q_{(m)}^\circ\right|
$$
holds for $R=[q,o,p,x_{-1},\tilde{x}_0,x_1,\ldots,x_k]\in\mathcal{K}_{p,q}$. Now, we distinguish two cases:
\begin{itemize}
\item If $k\geq 2$, then we set $x_{k+1}$ to be the reflected image of $x_{-1}$ through $\ell_{p,q}$, and move $x_k$ into the position $\tilde{x}_k\in{\rm int}\,C_{p,q}$ parallel  to ${\rm aff}\{x_{k-1},x_{k+1}\}$ in a way such that $\tilde{x}_k$ is the reflected image of $\tilde{x}_{0}$ through $\ell_{p,q}$. The argument above based on Lemma~\ref{areasum-polygon-variation} proves that
$$
\left|Q'_{(m)}\right|+\left|Q'^\circ_{(m)}\right|>\left|R\right|+\left|R^\circ\right|>\left|Q_{(m)}\right|+\left|Q_{(m)}^\circ\right|
$$
holds for $Q'_{(m)}=[q,o,p,x_{-1},\tilde{x}_0,x_1,\ldots,x_{k-1},\tilde{x}_k,x_{k+1}]\in\widetilde{\mathcal{P}}^{(k+2)}_A$. As $k+2\leq m$, this contradiction verifies (i) in this case.

\item If $k=1$, then we still set $x_{k+1}=x_2$ to be the reflected image of $x_{-1}$ through $\ell_{p,q}$. Possibly choosing $x_{-1}$ and $\tilde{x}_0$ even closer to $x_0$, if we slightly move $x_1$ into the position $\tilde{x}_1\in{\rm int}\,C_{p,q}$ parallel  to ${\rm aff}\{\tilde{x}_{0},x_{2}\}$ in a way such that ${\rm aff}\,\{\tilde{x}_{0},\tilde{x}_1\}$ is parallel to ${\rm aff}\{x_{-1},x_2\}$, then 
$$
\left|\widetilde{Q}_{(m)}\right|+\left|\widetilde{Q}^\circ_{(m)}\right|>\left|R\right|+\left|R^\circ\right|>\left|Q_{(m)}\right|+\left|Q_{(m)}^\circ\right|
$$
holds for $\widetilde{Q}_{(m)}=[q,o,p,x_{-1},\tilde{x}_0,\tilde{x}_1,x_{2}]\in \mathcal{K}_{p,q}$ by Lemma~\ref{areasum-polygon-variation}. Therefore, if $\widetilde{Q}'_{(m)}\in\widetilde{\mathcal{P}}^{(3)}_A$ is the  Steiner symmetral of $\widetilde{Q}_{(m)}$ with respect to $\ell_{p,q}$, then \eqref{Steiner-symm-inC-eq} yields
$$
\left|\widetilde{Q}'_{(m)}\right|+\left|\widetilde{Q}'^\circ_{(m)}\right|\geq \left|\widetilde{Q}_{(m)}\right|+\left|\widetilde{Q}^\circ_{(m)}\right|>\left|Q_{(m)}\right|+\left|Q_{(m)}^\circ\right|,
$$
which contradiction verifies (i) in all cases.

\end{itemize}

Next, we prove (ii) again indirectly; namely, we suppose that there exists $i\in\{1,\ldots,k-1\}$ such that the points $o$, $x_i$ and $\mu_i=(x_{i-1}+x_{i+1})/2$ are not collinear, and seek a contradiction. Here we may assume that $i<k/2$ by the symmetry through $\ell_{p,q}$, and hence $k\geq i+2$.  Let $s_i\neq \mu_i$ be the intersection point of $[o,x_i]$ and  $[x_{i-1},x_{i+1}]$. 
We slightly move $x_i$  parallel to ${\rm aff}\{x_{i-1},x_{i+1}\}$ into a position $\tilde{x}_i\in{\rm int}\,C_{p,q}$ in a way such that  the intersection $\tilde{s}_i$ of $[o,\tilde{x}_i]$ and  $[x_{i-1},x_{i+1}]$ lies in $(s_i,\mu_i)$, and both $x_{i-1}$ and $x_{i+1}$ stay vertices of 
$$
R=[q,o,p,x_{0},\ldots,x_{i-1},\tilde{x}_i,x_{i+1},\ldots,x_k]\in\mathcal{K}_{p,q}.
$$
We deduce from Lemma~\ref{areasum-polygon-variation} that 
$$
\left|R\right|+\left|R^\circ\right|>\left|Q_{(m)}\right|+\left|Q_{(m)}^\circ\right|.
$$
Again, we distinguish two cases:
\begin{itemize}
\item If $k-i>i+1$, then we  move $x_{k-i}$ into the position $\tilde{x}_{k-i}\in{\rm int}\,C_{p,q}$ parallel  to ${\rm aff}\{x_{k-i-1},x_{k-i+1}\}$ in a way such that $\tilde{x}_{k-i}$ is the reflected image of $\tilde{x}_{i}$ through $\ell_{p,q}$. The argument above based on Lemma~\ref{areasum-polygon-variation} proves that
$$
\left|Q'_{(m)}\right|+\left|Q'^\circ_{(m)}\right|>\left|R\right|+\left|R^\circ\right|>\left|Q_{(m)}\right|+\left|Q_{(m)}^\circ\right|
$$
holds for $Q'_{(m)}=[q,o,p,x_{0},\ldots,x_{i-1},\tilde{x}_i,x_{i+1},\ldots,x_{k-i-1},\tilde{x}_{k-i},x_{k-i+1},\ldots,x_k]\in\widetilde{\mathcal{P}}^{(m)}$. This contradiction verifies (ii) in this case.

\item Let $k-i=i+1$. After possibly choosing  $\tilde{x}_i$ even closer to $x_i$, if we move $x_{i+1}$ into the position $\tilde{x}_{i+1}\in{\rm int}\,C_{p,q}$ parallel  to ${\rm aff}\{\tilde{x}_{i},x_{i+2}\}$ in a way such that ${\rm aff}\,\{\tilde{x}_{i},\tilde{x}_{i+1}\}$ is parallel to ${\rm aff}\{x_{i-1},x_{i+2}\}$, and $\tilde{x}_{i}$ and $x_{i+2}$ are vertices of 
$$
\widetilde{Q}_{(m)}=[q,o,p,,x_{0},\ldots,x_{i-1},\tilde{x}_i,\tilde{x}_{i+1},x_{i+2},\ldots,x_k]\in \mathcal{K}_{p,q},
$$
then Lemma~\ref{areasum-polygon-variation} yields that
$$
\left|\widetilde{Q}_{(m)}\right|+\left|\widetilde{Q}^\circ_{(m)}\right|>\left|R\right|+\left|R^\circ\right|>\left|Q_{(m)}\right|+\left|Q_{(m)}^\circ\right|.
$$
 Therefore, if $\widetilde{Q}'_{(m)}\in\widetilde{\mathcal{P}}^{(k)}$ is the  Steiner symmetral of $\widetilde{Q}_{(m)}$ with respect to $\ell_{p,q}$, then \eqref{Steiner-symm-inC-eq} yields
$$
\left|\widetilde{Q}'_{(m)}\right|+\left|\widetilde{Q}'^\circ_{(m)}\right|\geq \left|\widetilde{Q}_{(m)}\right|+\left|\widetilde{Q}^\circ_{(m)}\right|>\left|Q_{(m)}\right|+\left|Q_{(m)}^\circ\right|,
$$
which contradiction verifies (ii) in all cases.

\end{itemize}

Finally, (ii) for $Q_{(m)}^\circ$ implies (iii) for $Q_{(m)}$ by \eqref{point-line-polarity}.\\

\noindent{\bf Step 2.} {\it Intersections with ellipses are candidates for extremality.}

We recall that we have assumed that $p$ and $q$ are orthogonal based on Claim~\ref{pq-orthogonal}, and hence $\|r\|=\sqrt{2}$ and $(p+q)/2=r/2$.

Having proved all the three properties (i), (ii) and (iii) of $Q_{(m)}$ for $m\geq 6$, the symmetric role of $Q_{(m)}$ and $Q_{(m)}^\circ$ allows us to assume that $Q_{(m)}$ is of type B (cf. \eqref{type-AB-polar}), thus $p$ and $q$ are proper vertices of $Q_{(m)}$.
 Setting
$$
\Xi=\max_{M\in \mathcal{K}_{p,q}}\left(|M|+|M^\circ|\right),
$$
we deduce from Claim~\ref{Steiner-symm-inC} that
 \begin{equation} 
\label{optM-symm}
\Xi=\max_{M\in \mathcal{K}_{p,q}\atop {\rm symmetric}\,{\rm through}\,\ell_{p,q}}\left(|M|+|M^\circ|\right).
\end{equation}
Now, any $M\in \mathcal{K}_{p,q}$ symmetric through $\ell_{p,q}$ can be arbitrary well approximated by polygons in $\mathcal{K}_{p,q}$ symmetric through $\ell_{p,q}$.
 Therefore,  the definition of $Q_{(m)}$ yields that
 \begin{equation} 
\label{optM-Qm}
\Xi=\lim_{m\to \infty}\left(|Q_{(m)}|+|Q_{(m)}^\circ|\right).
\end{equation}
It follows from Claim~\ref{C-not-optimal} that there exists $\varepsilon>0$ and points $\xi_0,\xi\in (r,(p+q)/2)$ such that
for any $M\in\mathcal{K}_{p,q}$ symmetric through $\ell_{p,q}$, we have
$$
|M|+|M^\circ|>\Xi-\varepsilon\mbox{ \ implies \ }\xi_0\in M\mbox{ \ and \ }\xi\not\in M.
$$
We deduce from \eqref{optM-Qm} that if $m$ is large, then
\begin{equation}
\label{Qm-nor-degenerate}
\xi_0\in Q_{(m)}\mbox{ \ and \ }\xi\not\in Q_{(m)}.
\end{equation}

We recall that $\sigma=\{tp+sq:s,t\geq 0\}$.
It follows from (ii) and (iii) that for any $m\geq 6$, there exists an $o$-symmetric ellipse $E_{(m)}$ symmetric through $\ell_{p,q}$ such that all proper vertices of $Q_{(m)}$ (including $p$ and $q$) lie on $\partial E_{(m)}$, and they are equally spaced according to the metric determined by $E_{(m)}$. In addition,   the sides of $Q_{(m)}^\circ$ meeting ${\rm int}\,\sigma$ touch $E_{(m)}^*$ in their midpoints.

We deduce from \eqref{Qm-nor-degenerate} that if $m$ is large, then
\begin{equation}
\label{Em-nor-degenerate}
\xi_0\in E_{(m)}\mbox{ \ and \ }r\not\in E_{(m)}.
\end{equation}
For the orthonormal basis $u,v$ of $\R^2$ where $u=r/\sqrt{2}$ and $\langle v,p\rangle=\sqrt{2}/2$, we consider the equation for the boundary $\partial E_{(m)}$. 
According to \eqref{Em-nor-degenerate} and $p,q\in\partial E_{(m)}$, there exist $\sqrt{2}/2<\|\xi_0\|\leq a_m<\sqrt{2}$ and some $b_m>\sqrt{2}/2$ such that $xu+yv\in \partial E_{(m)}$ if and only if
$$
\frac{x^2}{a_m^2}+\frac{y^2}{b_m^2}=1\mbox{ \ where \ }\frac{1}{2a_m^2}+\frac{1}{2b_m^2}=1.
$$
Here $a_m\geq \|\xi_0\| >\sqrt{2}/2$ yield that the sequence $\{b_m\}$ is also bounded; therefore, there exists a subsequence $\{E_{(m')}\}$ that tends to an ellipse $E$ in the Hausdorff metric, and we claim that $E$ satisfies that
\begin{description}
\item{(a)} $E$ is $o$-symmetric,  symmetric through $\ell_{p,q}$ and $p,q\in\partial E$,
\item{(b)} $E\cap\sigma=E\cap C_{p,q}\in \mathcal{K}_{p,q}$,
\item{(c)} $|E\cap C_{p,q}|+\big|(E\cap C_{p,q})^\circ\big|=\Xi$ .
\end{description}
Here (a) follows from the corresponding properties of each $E_{(m)}$. 
Since  the proper vertices of $Q_{(m)}$ lie on $\partial E_{(m)}$, and they are equally spaced according to the metric determined by $E_{(m)}$, the property (i) of $Q_{(m)}$ implies that
\begin{equation}
\label{Em-difference}
\lim_{m\to \infty}\Big|\left(\sigma\cap E_{(m)}\right)\Delta Q_{(m)}\Big|=\lim_{m\to \infty}\left(\big|\sigma\cap E_{(m)}\big|-|Q_{(m)}|\right)=0.
\end{equation}
Using \eqref{limit} and \eqref{Em-difference}, we deduce (b) from $Q_{(m)}\subset C_{p,q}$, and (c) from \eqref{relative-polar-limit} and  \eqref{optM-Qm}. \\

\noindent{\bf Step 3.} {\it Search for the optimal ellipse.}

Since $p$ and $q$ are assumed to be orthogonal by Claim~\ref{pq-orthogonal}, Proposition~\ref{BS-inC} will follow if we prove for the ellipse $E$ in (a), (b) and (c) that
 \begin{equation} 
\label{ellipse-ineq}
\big|C_{p,q}\cap E\big|+\big|(C_{p,q}\cap E)^\circ\big|\leq \frac{\pi}2.
\end{equation}

We use the orthonormal basis $u,v$ of $\R^2$ as in Step~2.
Let $a,b>0$ be the half axes of $E$ where (cf. (a)) $au\in\partial E$ and $bv\in \partial E$, and (a) and (b) imply that 
$\sqrt{2}/2<a\leq 1\leq b$. In particular, $E=\Phi B^2$ for a diagonal matrix $\Phi$
with entries $a$ and $b$ on the diagonal, and 
$$
C_{p,q}\cap E=\Phi\aleph
$$ 
where $\aleph=\{te+sf\in B^2:s,t\geq 0\}$ for $e=u\cos\frac{\beta}2+v\sin\frac{\beta}2$, $f=u\cos\frac{\beta}2-v\sin\frac{\beta}2$ and $\beta\in(0,\pi/2]$, and $\beta$ satisfies
$$
a\cos\frac{\beta}2=b\sin\frac{\beta}2=\frac1{\sqrt{2}}\mbox{ \ and \ }\tan\frac{\beta}2=\frac{a}b\leq 1\mbox{ \ and \ }
e=\frac{u}{a\sqrt{2}}+\frac{v}{b\sqrt{2}}.
$$
It follows from \eqref{point-line-polarity} that
$$
(C_{p,q}\cap E)^\circ=\left[o,p,\Phi^{-t}e\right]\cup \Phi^{-t}\aleph\cup \left[o,q,\Phi^{-t}f\right];
$$ 
therefore,
$$
\big|C_{p,q}\cap E\big|+\big|(C_{p,q}\cap E)^\circ\big|=\frac{ab\beta}2+
\frac{\beta}{2ab}+\frac1{2a^2}-\frac1{2b^2}.
$$
As $\beta\in(0,\frac{\pi}2]$, we can parametrize it by $t=\cos\beta\in[0,1)$, and consider
$$
f(t)=\big|C_{p,q}\cap E\big|+\big|(C_{p,q}\cap E)^\circ\big|=t+\frac{2-t^2}{\sqrt{1-t^2}}\;{\rm arc}\,{\rm tan}\sqrt{\frac{1-t}{1+t}}.
$$
We note that ${\rm arc}\,{\rm tan}\, s<s$ for $s\in(0,1)$. Therefore, if $t\in (0,1)$, then
$$
f'(t)=\frac{-t^2}{2(1-t^2)}+\frac{t^3}{(1-t^2)^{3/2}}
\;{\rm arc}\,{\rm tan}\sqrt{\frac{1-t}{1+t}}<
\frac{-t^2}{2(1-t^2)}\left(1-\frac{2t}{1+t}\right)<0.
$$
We conclude that
$$
\big|C_{p,q}\cap E\big|+\big|(C_{p,q}\cap E)^\circ\big|=f(t)\leq f(0)=\pi/2,
$$
proving \eqref{ellipse-ineq}, and in turn Proposition~\ref{BS-inC}.
\end{proof}

\section{Proof of Theorem~\ref{thm-volume-sum} and Theorem~\ref{BS-stab}}

\begin{proof}[Proof of Theorem~\ref{thm-volume-sum}]
Let $K\subset\R^2$ be a an $o$-symmetric convex body such that
$B^2\subset K$  is the John ellipse of $K$. According to Behrend \cite{Beh37}, p. 734, there exist a $k\in\{4,6\}$ and an $o$-symmetric set $\{u_1,\ldots,u_k\}\subset S^{1}\cap\partial K$ such that $u_1,\ldots,u_k$ are in clockwise order along $S^1$, and setting $u_0=u_k$, we have $\angle(u_{i-1},o,u_i)\leq\frac{\pi}2$ for $i=1,\ldots,k$. This property also follows from 
\eqref{John-cond} and \eqref{John-card}.

As the supporting line to $K$ at any $u_i$ has to be a supporting line of $B^2$, as well, we deduce that
$$
K\cap \sigma_i\subset C_{u_{i-1},u_i}\mbox{ \ and \ }K^*\cap \sigma_i=(K\cap \sigma_i)^{\circ}
$$
where $\sigma_i=\{su_{i-1}+tu_i:s,t\geq 0\}$ for $i=1,\ldots,k$, and $(K\cap \sigma_i)^{\circ}$ is defined as a set in
$\mathcal{K}_{u_{i-1},u_i}$. We conclude from Proposition~\ref{BS-inC} that
$$
|K\cap \sigma_i|+|K^*\cap \sigma_i|\leq \angle(u_{i-1},o,u_i)
$$
for $i=1,\ldots,k$; therefore, $|K|+|K^*|\leq 2\pi$. 

Finally, if $B^2\supset K$ is the L\"owner ellipse, then the proof is the same word by word.
\end{proof}
 
\begin{proof}[Proof of Theorem~\ref{BS-stab}]
Let $K\subset \R^2$ be an $o$-symmetric convex body satisfying $|K|\cdot |K^*|\geq (1-\varepsilon)\pi^2$ for $\varepsilon\in(0,\frac12)$.

First, we consider the case of  Theorem~\ref{BS-stab} concerning the John ellipse is $E_J$ of $K$ where we may asssume that $E_J=B^2$,  and hence $K^*\subset B^2\subset K$. It follows from Behrend's inequality \eqref{Behrend} that $|K|\leq 4$, thus $\sqrt{|K|}+\sqrt{\pi}<4$.
We deduce from $\sqrt{1-\varepsilon}>1-\varepsilon$, the AM-GM inequality and Theorem~\ref{thm-volume-sum}  that
\begin{equation}
\label{prod-sum-est}
(1-\varepsilon)2\pi< 2\sqrt{|K|\cdot |K^*|}\leq |K|+ |K^*|\leq 2\pi,
\end{equation} 
and hence using the formula $|K|+ |K^*|-2\sqrt{|K|\cdot |K^*|}=\left(\sqrt{|K|}-\sqrt{|K^*|}\right)^2$, we have
$$
2\pi\varepsilon >\left(\sqrt{|K|}-\sqrt{|K^*|}\right)^2\geq \left(\sqrt{|K|}-\sqrt{\pi}\right)^2=
\frac{\left(|K|-\pi\right)^2}{\left(\sqrt{|K|}+\sqrt{\pi}\right)^2}\geq \frac{\left(|K|-\pi\right)^2}{16};
$$
therefore, $|K|\geq \pi$ yields that
$$
|K\backslash B^2|\leq \sqrt{32\pi} \cdot\sqrt{\varepsilon}<4|K| \sqrt{\varepsilon}.
$$

Next, we consider the L\"owner ellipse $E_L$ of $K$ where we may asssume that $E_L=B^2$, and hence $K\subset B^2\subset K^*$.
 As \eqref{prod-sum-est} holds again, we have
$$
2\pi\varepsilon >\left(\sqrt{|K^*|}-\sqrt{|K|}\right)^2\geq \left(\sqrt{\pi}-\sqrt{|K|}\right)^2=
\frac{\left(\pi-|K|\right)^2}{\left(\sqrt{|K|}+\sqrt{\pi}\right)^2}\geq \frac{\left(\pi-|K|\right)^2}{4\pi};
$$
therefore, as Behrend's inequality \eqref{Behrend} yields that $|K|\geq 2$, we have
$$
|B^2\backslash K|\leq \sqrt{8\pi^2} \cdot\sqrt{\varepsilon}< 5|K| \sqrt{\varepsilon},
$$
completing the proof of Theorem~\ref{BS-stab}.
\end{proof}

\noindent{\bf Acknowledgement:} B\"or\"oczky's research is supported in part by NKKP grant 150613. Makai, Jr.'s research was partially supported by
several OTKA grants, and was supported by ERC Advanced Grant "Geoscape"
No. 882971.

\noindent K\'aroly J. B\"or\"oczky, HUN-REN Alfr\'ed R\'enyi Institute of Mathematics, Budapest, Hungary\\
boroczky.karoly.j@renyi.hu and
 ELTE,  Institute of Mathematics, Budapest, Hungary\\

\noindent Endre Makai, Jr., HUN-REN Alfr\'ed R\'enyi Institute of Mathematics, Budapest, Hungary\\
makai.endre@renyi.hu\\

\end{document}